\documentclass[a4paper,twoside,english,DIV12]{scrartcl}
\usepackage[T1]{fontenc}
\usepackage{verbatim}
\usepackage{mathtools}
\usepackage{amsmath}
\usepackage{amsthm}
\usepackage{amssymb}

\makeatletter

\numberwithin{equation}{section}
\numberwithin{figure}{section}
 \theoremstyle{definition}
 \newtheorem*{defn*}{\protect\definitionname}
\theoremstyle{plain}
\newtheorem{thm}{\protect\theoremname}[section]
  \theoremstyle{remark}
  \newtheorem{rem}[thm]{\protect\remarkname}
  \theoremstyle{plain}
  \newtheorem{lem}[thm]{\protect\lemmaname}

\usepackage[english]{babel}     % Deutsche Silbentrennung
\usepackage[T1]{fontenc}    % Silbentrennung auch nach Umlauten mglich
\usepackage{enumerate}
\usepackage{a4wide}        % A4-Format
\usepackage{amsmath}
\usepackage{url}
\usepackage{euscript}
\allowdisplaybreaks[1] % allow page breaks in some displayed equations
\usepackage{amsfonts}

\newcommand{\hide}[1]{}

\newcommand*{\dive}{\operatorname{div}}

\newcommand*{\curl}{\operatorname{curl}}

\newcommand*{\Grad}{\operatorname{Grad}}
\newcommand*{\Div}{\operatorname{Div}}

\newcommand*{\grad}{\operatorname{grad}}

\DeclareMathAccent{\Circ}{\mathalpha}{operators}{"17}
\newcommand{\interior}[1]{\Circ{#1}}

\renewcommand{\Re}{\operatorname{\mathfrak{Re}}}

\newcommand{\oi}[2]{\left]#1,#2 \right[}

\renewcommand{\tilde}{\widetilde}
\renewcommand*{\epsilon}{\varepsilon}
\renewcommand*{\rho}{\varrho}

\makeatother

\usepackage{babel}
  \providecommand{\definitionname}{Definition}
  \providecommand{\lemmaname}{Lemma}
  \providecommand{\remarkname}{Remark}
\providecommand{\theoremname}{Theorem}

\begin{document}
\title{On Wellposedness for Some Thermo-Piezo-Electric Coupling Models}

\author{T. Mulholland, R. Picard, S. Trostorff \& M. Waurick}

\maketitle
\textbf{Abstract.} {There is an increasing reliance on mathematical modelling to assist in the design of piezoelectric ultrasonic
transducers since this provides a cost-effective and quick way to arrive at a first prototype. Given a desired
operating envelope for the sensor the inverse problem of obtaining the associated design parameters within
the model can be considered. It is therefore of practical interest to examine the well-posedness of such
models. There is a need to extend the use of such sensors into high temperature environments and so this
paper shows, for a broad class of models, the well-posedness of the magneto-electro-thermo-elastic problem.
Due to its widespread use in the literature, we also show the well-posedness of the quasi-electrostatic case}

\textbf{MSC class:} 35; 35Q60; 46N20;74F05


%\date{}
\maketitle

\section{Introduction.}

Piezoelectric structures can receive electrical energy and use this
to alter their mechanical and thermal properties (and vice versa)
and hence they are ideally placed for use in smart materials that
can be automatically adjusted to assist in the vibration and thermal
stress control of structures \cite{rafiee2013}, as actuators \cite{tucsnak1996},
or as energy harvesters \cite{wynn2013}. To assist in the design
of such technology there is an increasing reliance on mathematical
models as a cost-effective and fast way to arrive at a first prototype
\cite{yuan2011}. The primary use of these piezoelectric materials
is in ultrasonic transducers. In transmission mode an electrical signal
is sent in to the device and the resulting mechanical vibration causes
an ultrasound wave to be transmitted through the material of interest.
Having traversed this material this mechanical wave is then received
by the transducer and converted back to an electrical signal for processing.
Typical applications of this technology can be found in medical imaging
and non-destructive testing of safety critical structures \cite{tant2015a,tant2015b,tramontana2015}.
The well-posedness of these models is therefore of practical interest
if these models are to be used in considering inverse problems centred
on optimising the material parameters and other design parameters
to meet some pre-specified operational quality of the device \cite{algehyne2015,mulholland2007,mulholland2008,orr2008a,orr2008b,orr2007,walker2010,walker2011,walker2015}.
The well-posedness of the forward problem in piezoelectric material
modelling is also essential if one wants to consider the inverse problem
of obtaining the piezoelectric material tensors from experimental
measurements \cite{kaltenbacher2006,lahmer2008}. A modelling assumption
that is often employed in order to reduce the size of the model is
the so-called quasi-electrostatic approach \cite{daros2000,fang2013,imperiale2012,leugering2015}.
Since the electrical waves travel at a far faster speed than the mechanical
waves within a piezoelectric material then this separation in time
scales is used to focus attention on the timescale of the mechanical
waves by assuming that the electrical activity is instantaneous (so
the electrical potential is spatially uniform). As well as simplifying
the analysis this approach can also help to reduce the computational
time in numerical simulations of electromechanical waves propagating
in a piezoelectric material \cite{jiang2000}. Given the widespread
use of this approximation the question of the well-posedness of the
quasi-electrostatic model of a piezoelectric material has been studied
\cite{kapitonov2003,kapitonov2006,kapitonov2007}. It is also possible
to incorporate dissipative loss terms in the formulation and consider
the existence and regularity of the model solution \cite{miara2009}
and thermal effects \cite{mindlin1974}. A useful summary of the literature
that has examined this well-posedness for a range of boundary conditions
is provided by Akamatsu and Nakamura \cite{akamatsu2002}. They prove
that in a bounded domain with a Lipschitz boundary the initial-boundary
value problem is well-posed. The initial conditions pertain to the
mechanical displacement and its time derivative, and boundary conditions
are stated in terms of the mechanical displacement, the mechanical
stress, the electric potential and the electric displacement; this
is a typical scenario for a piezoelectric transducer operating in
reception mode. When a piezoelectric device is being tested it is
normally immersed in a water bath and its transmission performance
assessed by using a receiving hydrophone at some distance from the
piezoelectric device. The existence and uniqueness of solutions to
this fluid-solid interaction problem have very recently been settled
using a potential method and the theory of pseudodifferential equations
\cite{chkadua2015}. The situation where the piezoelectric material
is in contact with a surface and undergoing anti-plane motion has
also been investigated \cite{ionica2013,migorski2009}. In this situation
the uniqueness and existence of a weak solution is proved. The continuous
dependence of the solution on the data has also been shown \cite{barboteu2009}
and this has been used to develop a numerical solution of the problem.
This piezoelectric problem was recently extended to consider the existence
and uniqueness when coupled with the thermal effects \cite{benaissa2015}.
Such models are needed to assist in the design of smart ceramic materials
that can adjust their mechanical properties as the temperature
of their environment fluctuates; in the non-destructive testing application,
for example, there is an increasing demand for ultrasonic transducers
that can withstand high temperatures. The above work is based on the
quasi-electrostatic approximation and the well-posedness of the non-stationary
piezoelectric system, wherein the Maxwell equations are involved,
has also been considered \cite{ammari2010}; one simplification that
is often used however is to study the time harmonic case \cite{mercier2005}.
The emergence of smart materials has very recently led researchers
to extend the model to include magnetic and thermal effects and the
well-posedness of the magneto-electro-thermo-elastic problem has been
very recently considered \cite{bonaldi2015,perla2014,sixto2013}.
By non-dimensionalising the problem and identifying a small parameter
this dynamic model was systematically reduced to the quasi-electrostatic
case. Note that they assumed that all the coefficients were bounded and satisfied some conditions
warranting positive definiteness.  This then permitted the use of semi-group theory to
settle the well-posedness issue. The resulting electrical and magnetic
potentials enabled consideration of the energy transfer in the material
and this therefore has applications in the design of energy harvesting materials.
In this paper we will consider a similar model but our methodology
will apply to a broader class of problems (for example the operators
here could be non-local of convolution type) and will not be restricted
to the multiplication operators as in this previous work. The coefficients
could be fully anisotropic, inhomogeneous and even could -- for example
-- be non-local. In some cases material behaviour may be better described
by an averaging influence of the neighborhood. To be concrete, operator
coefficients could show a behaviour like
\[
u\mapsto Ku
\]
with
\[
\left(Ku\right)\left(x\right)\coloneqq\int_{w}k\left(x-y\right)\:u\left(x\right)\:dx,
\]
where $k$ is a kernel function, e.g. generated by the Gaussian distribution,
and $x+\left[w\right]$, $w\subseteq\mathbb{R}^{3}$, a window of
influence impacting on the behaviour at point $x$. The considerations
are limited in so far as only linear material behaviour is considered.
Thus, for example, the non-linear effects of the temperature, which are
known to limit the piezo-electric effect, see the plate case
e.g. \cite{Ciarlet97I}, \cite{Rahmoune01121998} and the references
cited therein, are clearly beyond the scope of this ground-laying study
of the general linear case.

\section{The System Equations of Thermo-Piezo-Electro-Magnetism.}

Let $\Omega\subseteq\mathbb{R}^{3}$. The system of Thermo-Piezo-Electro-Magnetism
is a coupled system consisting of the equation of elasticity, Maxwell's
equations and the equation of heat conduction. Throughout, we denote
by $\partial_{0}$ the derivative with respect to time. The equation
of elasticity is given by
\begin{align}
\partial_{0}^{2}\rho_{\ast}u-\Div T & =F_{0},\label{eq:elasticity}
\end{align}
where $u:\mathbb{R}\times\Omega\to\mathbb{R}^{3}$ describes the displacement
of the elastic body $\Omega,$ $T:\mathbb{R}\times\Omega\to\mathrm{sym}\left[\mathbb{R}^{3\times3}\right]$
denotes the stress tensor, which is assumed to attain values in the space
of symmetric matrices. The function $\rho_{\ast}:\Omega\to\mathbb{R}$
stands for the density of $\Omega$ and $F_{0}:\mathbb{R}\times\Omega\to\mathbb{R}^{3}$
is an external force term. Maxwell's equation are given by
\begin{align}
\partial_{0}B+\curl E & =F_3,\nonumber \\
\partial_{0}D-\curl H & =F_2-\sigma E.\label{eq:Maxwell}
\end{align}
Here, $B,D,E,H:\mathbb{R}\times\Omega\to\mathbb{R}^{3}$ denote the
magnetic flux density, the electric displacement, the electric field
and the magnetic field, respectively. The functions $F_2,F_3:\mathbb{R}\times\Omega\to\mathbb{R}^{3}$
are given source terms and $\sigma:\Omega\to\mathbb{R}$ denotes the
resistance. Finally, the equation of heat conduction is given by
\begin{equation}
\partial_{0}\Theta_{0}\eta+\dive q=F_{4},\label{eq:heat}
\end{equation}
where $\eta:\mathbb{R}\times\Omega\to\mathbb{R}$ denotes the entropy
density, $q:\mathbb{R}\times\Omega\to\mathbb{R}^{3}$ is the heat
flux, $F_{4}:\mathbb{R}\times\Omega\to\mathbb{R}$ is an external
heat source and $\Theta_{0}:\Omega\to\mathbb{R}$ is the reference temperature.
Of course, all these equations need to be completed by suitable material
laws, where also the coupling will occur. As it will turn out, the
system can be written in the following abstract form
\begin{equation}
\left(\partial_{0}M_{0}+M_{1}+\left(\begin{array}{cccccc}
0 & -\Div & 0 & 0 & 0 & 0\\
-\Grad & 0 & 0 & 0 & 0 & 0\\
0 & 0 & 0 & -\curl & 0 & 0\\
0 & 0 & \curl & 0 & 0 & 0\\
0 & 0 & 0 & 0 & 0 & \dive\\
0 & 0 & 0 & 0 & \grad & 0
\end{array}\right)\right)\left(\begin{array}{c}
v\\
T\\
E\\
H\\
\Theta_{0}^{-1}\theta\\
q
\end{array}\right)=F,\label{eq:system}
\end{equation}
for suitable bounded operators $M_{0},M_{1}$ on the Hilbert space
$H\coloneqq L^{2}(\Omega)^{3}\oplus\mathrm{sym}\left[L^{2}(\Omega)^{3\times3}\right]\oplus L^{2}(\Omega)^{3}\oplus L^{2}(\Omega)^{3}\oplus L^{2}(\Omega)\oplus L^{2}(\Omega)^{3}$
. Here, $v\coloneqq\partial_{0}u$ and $\theta:\mathbb{R}\times\Omega\to\mathbb{R}$
denotes the temperature. Of course, we also need to impose boundary
conditions. To make this precise, we need to define the spatial differential
operators.
\begin{defn*}
We denote by $\interior C_{\infty}(\Omega)$ the space of arbitrarily
differentiable functions with compact support in $\Omega.$ Then we
define the operator $\interior\grad$ as the closure of
\begin{align*}
\interior C_{\infty}(\Omega)\subseteq L^{2}(\Omega) & \to L^{2}(\Omega)^{3}\\
\phi & \mapsto\left(\partial_{1}\phi,\partial_{2}\phi,\partial_{3}\phi\right)
\end{align*}
as well as $\interior\dive$ as the closure of
\begin{align*}
\interior C_{\infty}(\Omega)^{3}\subseteq L^{2}(\Omega)^{3} & \to L^{2}(\Omega)\\
(\phi_{1},\phi_{2},\phi_{3}) & \mapsto\sum_{i=1}^{3}\partial_{i}\phi_{i}.
\end{align*}
Integration by parts shows $\interior\dive\subseteq-(\interior\grad)^{\ast}$
and we set $\dive\coloneqq-\left(\interior\grad\right)^{\ast}$ and
$\grad\coloneqq-\left(\interior\dive\right)^{\ast}.$ Similarly, we
define the operator $\interior\curl$ as the closure of
\begin{align*}
\interior C_{\infty}(\Omega)^{3}\subseteq L^{2}(\Omega)^{3} & \to L^{2}(\Omega)^{3}\\
(\phi_{1},\phi_{2},\phi_{3}) & \mapsto\left(\begin{array}{ccc}
0 & -\partial_{3} & \partial_{2}\\
\partial_{3} & 0 & -\partial_{1}\\
-\partial_{2} & \partial_{1} & 0
\end{array}\right)\left(\begin{array}{c}
\phi_{1}\\
\phi_{2}\\
\phi_{3}
\end{array}\right)
\end{align*}
and $\curl\coloneqq\left(\interior\curl\right)^{\ast}\supseteq\interior\curl.$
Finally, we define $\interior\Grad$ and $\interior\Div$ as the closure
of
\begin{align*}
\interior C_{\infty}(\Omega)^{3}\subseteq L^{2}(\Omega)^{3} & \to\mathrm{sym}\left[L^{2}(\Omega)^{3\times3}\right]\\
(\phi_{1},\phi_{2},\phi_{3}) & \mapsto\frac{1}{2}\left(\partial_{j}\phi_{i}+\partial_{i}\phi_{j}\right)_{i,j\in\{1,2,3\}}
\end{align*}
and of
\begin{align*}
\mathrm{sym}\left[\interior C_{\infty}(\Omega)^{3\times3}\right]\subseteq\mathrm{sym}\left[L^{2}(\Omega)^{3\times3}\right] & \to L^{2}(\Omega)^{3}\\
(\phi_{ij})_{i,j\in\{1,2,3\}} & \mapsto\left(\sum_{j=1}^{3}\partial_{j}\phi_{ij}\right)_{i\in\{1,2,3\}},
\end{align*}
respectively and set $\Grad\coloneqq-\left(\interior\Div\right)^{\ast}$
as well as $\Div\coloneqq-\left(\interior\Grad\right)^{\ast}.$ Elements
in the domain of the operators marked by a circle satisfy an abstract
homogeneous boundary condition, which, in the case of a smooth boundary
$\partial\Omega$, can be written as
\[
u=0\mbox{ on }\partial\Omega
\]
for $u\in D(\interior\grad)$ or $u\in D(\interior\Grad)$,
\[
u\cdot n=0\mbox{ on }\partial\Omega
\]
for $u\in D(\interior\dive)$ or $u\in D(\interior\Div),$ where $n$
denotes the exterior unit normal vector field on $\partial\Omega$
and
\[
u\times n=0\mbox{ on }\partial\Omega,
\]
for $u\in D(\interior\curl).$
\end{defn*}
We will assume that $v=0,E\times n=0$ and $q\cdot n=0$ on the boundary
and hence, the block operator in (\ref{eq:system}) will be replaced
by
\[
\left(\begin{array}{cccccc}
0 & -\Div & 0 & 0 & 0 & 0\\
-\interior\Grad & 0 & 0 & 0 & 0 & 0\\
0 & 0 & 0 & -\curl & 0 & 0\\
0 & 0 & \interior\curl & 0 & 0 & 0\\
0 & 0 & 0 & 0 & 0 & \interior\dive\\
0 & 0 & 0 & 0 & \grad & 0
\end{array}\right),
\]
which is now a skew-selfadjoint operator on $H$.

To recall the solution theory (as described in the last chapter of
\cite{PDE_DeGruyter}) for our simple situation the needed requirement
is that $M_{0}$ is selfadjoint and that
\begin{equation}
\nu M_{0}+\Re M_{1}\geq c_{0}>0\mbox{ for all sufficiently large }\nu\in]0,\infty[.\:\label{eq:posdef}
\end{equation}
Equation (\ref{eq:posdef}) is for example satisfied if $M_{0}$
is strictly positive definite on its range and $\Re M_{1}$ strictly
positive definite on the null space of $M_{0}$.
\begin{rem}
Whenever we are not interested in the actual constant $c_{0}\in\oi0\infty$
we shall write for the strict positive definiteness constraint
\[
\Re T\geq c_{0}
\]
simply
\[
T\gg0.
\]
So, the general requirement for the problem class under consideration
would be written as
\begin{equation}
M_{0}\:\mbox{selfadjoint},\:\nu M_{0}+M_{1}\gg0\label{eq:pos-def11}
\end{equation}
for all sufficiently large $\nu\in]0,\infty[$~.
\end{rem}

\section{The System Equations of Thermo-Piezo-Electricity.}

In this section we discuss material relations suggested in \cite{mindlin1974}
and derive the structure of the corresponding operators $M_{0}$ and
$M_{1}$. Furthermore, we give sufficient conditions on the parameters
involved to warrant the solvability condition (\ref{eq:pos-def11}).  The material relations described in \cite{mindlin1974} are initially
given in the form (where we write $\mathcal{E}=\Grad u$ as usual
for the strain tensor)
\begin{eqnarray*}
T & = & C\:\mathcal{E}-eE-\lambda\theta,\\
D & = & e^{*}\mathcal{E}+\epsilon E+p\,\theta,\\
B & = & \mu\,H,\\
\eta & = & \lambda^{*}\mathcal{E}+p^{*}E+\alpha\,\Theta_{0}^{-1}\,\theta.
\end{eqnarray*}
Here $C\in L\left(\mathrm{sym}\left[L^{2}(\Omega)^{3\times3}\right]\right)$
is the elasticity tensor, $\varepsilon,\mu\in L\left(L^{2}(\Omega)^{3}\right)$
are the permittivity and permeability, respectively, $\alpha\coloneqq\rho_{\ast}c\in L(L^{2}(\Omega))$
is the product of the mass density\footnote{Throughout, we identify $L^{\infty}-$functions with their induced
multiplication operators.} $\rho_{\ast}\in L^{\infty}(\Omega)$ and the specific heat capacity
$c\in L(L^{2}(\Omega))$ and $\Theta_{0}:\Omega\to\mathbb{R}$ is
the reference temperature which satisfies $\Theta_{0},\Theta_{0}^{-1}\in L^{\infty}(\Omega)$
.The operators $e\in L\left(L^{2}(\Omega)^{3};\mathrm{sym}\left[L^{2}(\Omega)^{3\times3}\right]\right),\lambda\in L\left(L^{2}(\Omega);\mathrm{sym}\left[L^{2}(\Omega)^{3\times3}\right]\right),p\in L(L^{2}(\Omega);L^{2}(\Omega)^{3})$
are coupling parameters. As a first minor adjustment we make the relative
temperature $\Theta_{0}^{-1}\,\theta$ our new unknown temperature
function yielding
\begin{eqnarray*}
T & = & C\:\mathcal{E}-eE-\left(\lambda\Theta_{0}\right)\Theta_{0}^{-1}\theta,\\
D & = & e^{*}\mathcal{E}+\epsilon E+\left(p\Theta_{0}\right)\,\Theta_{0}^{-1}\theta,\\
B & = & \mu\,H,\\
\Theta_{0}\eta & = & \left(\Theta_{0}\lambda^{*}\right)\mathcal{E}+\left(\Theta_{0}p^{*}\right)E+\gamma_{0}\,\Theta_{0}^{-1}\theta,
\end{eqnarray*}
where we have introduced the abbreviation
\begin{align*}
\gamma_{0} & \coloneqq\Theta_{0}\alpha.
\end{align*}
We assume that heat conduction is governed by the Maxwell-Cattaneo-Vernotte
modification
\begin{align*}
\partial_{0}\kappa_{1}q+\kappa_{0}^{-1}q+\grad\theta & =0,
\end{align*}
for operators $\kappa_{0},\kappa_{1}\in L(L^{2}(\Omega)^{3})$. To
adapt the material relations to our framework we solve for $\mathcal{E}$
and obtain
\begin{eqnarray*}
\mathcal{E} & = & C^{-1}T+C^{-1}eE+C^{-1}\left(\lambda\Theta_{0}\right)\Theta_{0}^{-1}\theta,\\
D & = & e^{*}C^{-1}T+\left(\epsilon+e^{*}C^{-1}e\right)E+\left(p\Theta_{0}+e^{*}C^{-1}\lambda\Theta_{0}\right)\Theta_{0}^{-1}\theta,\\
B & = & \mu\:H,\\
\Theta_{0}\eta & = & \Theta_{0}\lambda^{*}C^{-1}T+\left(\Theta_{0}p^{*}+\Theta_{0}\lambda^{*}C^{-1}e\right)E+\left(\gamma_{0}+\Theta_{0}\lambda^{*}C^{-1}\lambda\Theta_{0}\right)\Theta_{0}^{-1}\theta.
\end{eqnarray*}
Thus, we arrive at a equation of the form (\ref{eq:system}) with
\[
M_{0}\coloneqq\left(\begin{array}{cccccc}
\rho_{*} & 0 & \quad0 & 0 & 0 & 0\\
0 & C^{-1} & \quad C^{-1}e & 0 & C^{-1}\lambda\Theta_{0} & 0\\
0 & e^{*}C^{-1} & \quad\left(\epsilon+e^{*}C^{-1}e\right) & 0 & \left(p\Theta_{0}+e^{*}C^{-1}\lambda\Theta_{0}\right) & 0\\
0 & 0 & \quad0 & \mu & 0 & 0\\
0 & \Theta_{0}\lambda^{*}C^{-1} & \quad\left(\Theta_{0}p^{*}+\Theta_{0}\lambda^{*}C^{-1}e\right) & \quad0\quad & \,\left(\gamma_{0}+\Theta_{0}\lambda^{*}C^{-1}\lambda\Theta_{0}\right) & 0\\
0 & 0 & \quad0 & 0 & 0 & \kappa_{1}
\end{array}\right)
\]
and
\[
M_{1}\coloneqq\left(\begin{array}{cccccc}
0 & 0 & 0 & 0 & 0 & 0\\
0 & 0 & 0 & 0 & 0 & 0\\
0 & 0 & \sigma & 0 & 0 & 0\\
0 & 0 & 0 & 0 & 0 & 0\\
0 & 0 & 0 & 0 & 0 & 0\\
0 & 0 & 0 & 0 & 0 & \kappa_{0}^{-1}
\end{array}\right).
\]
We need to verify the solvability condition (\ref{eq:pos-def11})
for these operators $M_{0}$ and $M_{1}.$
\begin{thm}
Assume that $\rho_{\ast},\varepsilon,\mu,C,\gamma_{0}$ are selfadjoint
and non-negative. Furthermore, we assume $\rho_{\ast},\mu,C,\gamma_{0}\gg0$
as well as $\nu\left(\varepsilon-\Theta_{0}p\gamma_{0}^{-1}p^{\ast}\Theta_{0}\right)+\sigma,\nu\kappa_{1}+\kappa_{0}^{-1}\gg0$
for sufficiently large $\nu>0.$ Then, $M_{0}$ and $M_{1}$ satisfy
the condition (\ref{eq:pos-def11}) and hence, the corresponding problem
of thermo-piezo-electricity is well-posed. \end{thm}
\begin{proof}
Obviously, $M_{0}$ is selfadjoint. Moreover, since $\rho_{\ast},\mu,\nu\kappa_{1}+\kappa_{0}^{-1}\gg0$
for sufficiently large $\nu,$ the only thing, which is left to show
is, that

\[
\nu\left(\begin{array}{ccc}
C^{-1} & C^{-1}e & C^{-1}\lambda\Theta_{0}\\
e^{\ast}C^{-1} & \varepsilon+e^{\ast}C^{-1}e & p\Theta_{0}+e^{\ast}C^{-1}\lambda\Theta_{0}\\
\Theta_{0}\lambda^{\ast}C^{-1}\quad & \Theta_{0}p^{\ast}+\Theta_{0}\lambda^{\ast}C^{-1}e\quad & \gamma_{0}+\Theta_{0}\lambda^{\ast}C^{-1}\lambda\Theta_{0}
\end{array}\right)+\left(\begin{array}{ccc}
0 & 0 & 0\\
0 & \sigma & 0\\
0 & 0 & 0
\end{array}\right)\gg0
\]
for sufficiently large $\nu.$ By symmetric Gauss steps as a congruence
transformations we get that the above operator is congruent to
\[
\nu\left(\begin{array}{ccc}
C^{-1} & 0 & 0\\
0 & \varepsilon & p\Theta_{0}\\
0 & \Theta_{0}p^{\ast} & \gamma_{0}
\end{array}\right)+\left(\begin{array}{ccc}
0 & 0 & 0\\
0 & \sigma & 0\\
0 & 0 & 0
\end{array}\right),
\]
which itself is congruent by another symmetric Gauss step to
\[
\nu\left(\begin{array}{ccc}
C^{-1} & 0 & 0\\
0 & \varepsilon-\Theta_{0}p\gamma_{0}^{-1}p^{\ast}\Theta_{0} & 0\\
0 & 0 & \gamma_{0}
\end{array}\right)+\left(\begin{array}{ccc}
0 & 0 & 0\\
0 & \sigma & 0\\
0 & 0 & 0
\end{array}\right).
\]

The latter operator is then strictly positive definite by assumption
and so the assertion follows.\end{proof}
\begin{rem}
~
\begin{enumerate}
\item Note that due to the generality of the assumptions then limit cases
such as $\epsilon=\Theta_{0}p\gamma_{0}^{-1}p^{\ast}\Theta_{0}$ and
$\sigma\gg0$ (eddy current case) are covered by the theorem.
\item To provide a hint towards further generalizations of the specific
model we incorporate for example piezo-magnetic effects by adding
in a corresponding coupling term. That is, we replace $M_{0}$ by
the operator{\footnotesize{} }
\[
\left(\begin{array}{cccccc}
\rho_{*}+\beta\mu\beta^{*} & 0 & \quad0 & -\beta\mu & 0 & 0\\
0 & C^{-1} & \quad C^{-1}e & 0 & C^{-1}\lambda\Theta_{0} & 0\\
0 & e^{*}C^{-1} & \quad\left(\epsilon+e^{*}C^{-1}e\right) & 0 & \left(p\Theta_{0}+e^{*}C^{-1}\lambda\Theta_{0}\right) & 0\\
-\mu\beta^{*} & 0 & \quad0 & \mu & 0 & 0\\
0 & \Theta_{0}\lambda^{*}C^{-1} & \quad\left(\Theta_{0}p^{*}+\Theta_{0}\lambda^{*}C^{-1}e\right) & \quad0\quad & \,\left(\gamma_{0}+\Theta_{0}\lambda^{*}C^{-1}\lambda\Theta_{0}\right) & 0\\
0 & 0 & \quad0 & 0 & 0 & \kappa_{1}
\end{array}\right),
\]
where $\beta\in L(L^{2}(\Omega)^{3})$ is a further parameter, which
now couples the displacement field $u$ with the magnetic field $H$.
Note that to match up with the piezo-magnetic model discussed in \cite{MR2318265}
we have to introduce a new composite magnetic field
\[
\tilde{H}\;:=\beta^{*}v+H
\]
in place of $H$ as one of the basic unknowns. The well-posedness
conditions remain unchanged.
\end{enumerate}
\end{rem}

\section{A ``Simplification''.}

The above situation is commonly ``simplified'' by replacing the
full Maxwell equations by the static Maxwell equations for the electric
field; the so called quasi-electrostatic approach. There is a price to be paid for this modification, which made
us use quotation marks around the term ``simplify `` and ``simplification''.
We assume that $E=-\interior\grad\varphi$ for a suitable potential
$\varphi\in D(\interior\grad)$ and $D\in D(\dive)$ and we set $\psi\coloneqq\dive D.$
Moreover, we assume that there is no conductivity term, i.e. $\sigma=0$
and thus, the system under consideration reduces to {\small{}
\begin{align}
\left(\partial_{0}\left(\begin{array}{cccc}
\rho_{*} & 0 & 0 & 0\\
0 & C^{-1} & C^{-1}\lambda\Theta_{0} & 0\\
0 & \Theta_{0}\lambda^{*}C^{-1} & \,\left(\gamma_{0}+\Theta_{0}\lambda^{*}C^{-1}\lambda\Theta_{0}\right) & 0\\
0 & 0 & 0 & \kappa_{1}
\end{array}\right)+\left(\begin{array}{cccc}
0 & 0 & 0 & 0\\
0 & 0 & 0 & 0\\
0 & 0 & 0 & 0\\
0 & 0 & 0 & \kappa_{0}^{-1}
\end{array}\right)+\left(\begin{array}{cccc}
0 & -\Div & 0 & 0\\
-\interior\Grad & 0 & 0 & 0\\
0 & 0 & 0 & \interior\dive\\
0 & 0 & \grad & 0
\end{array}\right)\right)\left(\begin{array}{c}
v\\
T\\
\Theta_{0}^{-1}\theta\\
q
\end{array}\right)+\nonumber \\
+\partial_{0}\left(\begin{array}{c}
\left(\begin{array}{c}
0\\
C^{-1}eE
\end{array}\right)\\
\left(\begin{array}{c}
\left(\Theta_{0}p^{*}+\Theta_{0}\lambda^{*}C^{-1}e\right)E\\
0
\end{array}\right)
\end{array}\right)=\left(\begin{array}{c}
F_{0}\\
F_{1}\\
F_{4}\\
F_{5}
\end{array}\right).\label{eq:quasi_static}
\end{align}
}We have now to express $E$ in terms of the other unknowns as part
of a new material law. Recall, that we have
\begin{align}
D & =e^{*}C^{-1}T+\left(\epsilon+e^{*}C^{-1}e\right)E+\left(p\Theta_{0}+e^{*}C^{-1}\lambda\Theta_{0}\right)\Theta_{0}^{-1}\theta.\label{eq:DandE}
\end{align}
By setting
\begin{align*}
\Phi & \coloneqq e^{*}C^{-1}T+\left(p\Theta_{0}+e^{*}C^{-1}\lambda\Theta_{0}\right)\left(\Theta_{0}^{-1}\theta\right)\\
 & =e^{*}C^{-1}\left(T+\lambda\theta\right)+p\theta,
\end{align*}
$D$ can be written as
\[
D=(\varepsilon+e^{\ast}C^{-1}e)E+\Phi.
\]

Using now that $\psi=\dive D$ and $E=-\interior\grad\varphi$ we
get that
\begin{align*}
-\dive\left(\epsilon+e^{*}C^{-1}e\right)\interior\grad\varphi+\dive\Phi & =\psi.
\end{align*}
We assume that $C,\varepsilon$ are selfadjoint and $\varepsilon+e^{\ast}C^{-1}e\gg0$
and set $M\coloneqq\sqrt{\varepsilon+e^{\ast}C^{-1}e}$. Then, the
latter equality can be written as
\[
-\dive M^{2}\:\interior\grad\varphi+\dive MM^{-1}\Phi=\psi,
\]
which gives
\[
\interior\grad\varphi+M^{-1}\left(\left(M\interior\grad\right)\left|M\:\interior\grad\right|^{-2}\dive M\right)M^{-1}\Phi=M^{-1}\left(M\interior\grad\right)\left|M\:\interior\grad\right|^{-2}\psi,
\]
if we assume that\footnote{Note that if $\Omega$ is bounded, this condition is always satisfied
since in this case $\interior\grad$ is onto.} $\psi\in D\left(\left|M\interior\grad\right|^{-2}\right).$ This
suggests to replace
\begin{align*}
E & =-\interior\grad\varphi\\
 & =M^{-1}\left(\overline{\left(M\interior\grad\right)\left|M\:\interior\grad\right|^{-2}\dive M}\right)M^{-1}\Phi-M^{-1}\left(M\interior\grad\right)\left|M\:\interior\grad\right|^{-2}\psi,
\end{align*}
where we use the closure bar to ensure this operator is in $L\left(L^{2}(\Omega)^{3}\right)$.
Indeed, this operator is not only bounded but also an orthogonal projector
as the next lemma shows.
\begin{lem}
Let $A:D\left(A\right)\subseteq H_{0}\to H_{1}$ be a densely defined
and closed linear operator between two Hilbert spaces $H_{0},H_{1}$
such that $A^{*}A$ is injective. Then
\[
\overline{A\left(A^{*}A\right)^{-1}A^{*}}=P_{\overline{A\left[H_{0}\right]}},
\]
the orthogonal projector on the closure of the range of $A$.\end{lem}
\begin{proof}
Let $x\in D(A^{\ast})$ and set
\[
f\coloneqq A\left(A^{\ast}A\right)^{-1}A^{\ast}x.
\]
Then, obviously, $f\in A\left[H_{0}\right]$ and $f\in D(A^{\ast}).$
Since $\overline{A[H_{0}]}=\left([\{0\}]A^{\ast}\right)^{\bot}$,
we get that
\[
A^{*}P_{\overline{A\left[H_{0}\right]}}x=A^{\ast}A\left(A^{\ast}A\right)^{-1}A^{\ast}x=A^{*}f
\]
and so
\[
A^{*}\left(P_{\overline{A\left[H_{0}\right]}}x-f\right)=0,
\]
which implies
\[
P_{\overline{A\left[H_{0}\right]}}x-f\in[\{0\}]A^{\ast}\cap\overline{A[H_{0}]}=\{0\}.
\]
Thus,
\[
P_{\overline{A\left[H_{0}\right]}}x=A\left(A^{\ast}A\right)^{-1}A^{\ast}x
\]
for each $x\in D(A^{\ast})$ and thus, the assertion follows by the
density of $D(A^{\ast}).$
\end{proof}
Using this result for $A\coloneqq M\interior\grad$ and setting $P\coloneqq P_{\overline{A\left[L^{2}(\Omega)\right]}}$
we get that
\[
E=-M^{-1}PM^{-1}\Phi-M^{-1}\left(M\interior\grad\right)\left|M\:\interior\grad\right|^{-2}\psi.
\]

Hence, the last term on the left hand side in (\ref{eq:quasi_static})
can be replaced by
\begin{align*}
 & \left(\begin{array}{c}
\left(\begin{array}{c}
0\\
C^{-1}eE
\end{array}\right)\\
\left(\begin{array}{c}
\left(\Theta_{0}p^{*}+\Theta_{0}\lambda^{*}C^{-1}e\right)E\\
0
\end{array}\right)
\end{array}\right)\\
= & -\left(\begin{array}{c}
\left(\begin{array}{c}
0\\
C^{-1}eM^{-1}PM^{-1}\Phi
\end{array}\right)\\
\left(\begin{array}{c}
\left(\Theta_{0}p^{*}+\Theta_{0}\lambda^{*}C^{-1}e\right)M^{-1}PM^{-1}\Phi\\
0
\end{array}\right)
\end{array}\right)-\left(\begin{array}{c}
\left(\begin{array}{c}
0\\
C^{-1}eM^{-1}\left(M\interior\grad\right)\left|M\interior\grad\right|^{-2}\psi
\end{array}\right)\\
\left(\begin{array}{c}
\left(\Theta_{0}p^{*}+\Theta_{0}\lambda^{*}C^{-1}e\right)M^{-1}\left(M\interior\grad\right)\left|M\interior\grad\right|^{-2}\psi\\
0
\end{array}\right)
\end{array}\right).
\end{align*}
Using now the definition of $\Phi,$ we can write
\begin{align*}
 & -\left(\begin{array}{c}
\left(\begin{array}{c}
0\\
C^{-1}eM^{-1}PM^{-1}\Phi
\end{array}\right)\\
\left(\begin{array}{c}
\left(\Theta_{0}p^{*}+\Theta_{0}\lambda^{*}C^{-1}e\right)M^{-1}PM^{-1}\Phi\\
0
\end{array}\right)
\end{array}\right)\\
= & -\left(\begin{array}{c}
\left(\begin{array}{c}
0\\
C^{-1}eM^{-1}PM^{-1}e^{\ast}C^{-1}T
\end{array}\right)\\
\left(\begin{array}{c}
\left(\Theta_{0}p^{*}+\Theta_{0}\lambda^{*}C^{-1}e\right)M^{-1}PM^{-1}e^{\ast}C^{-1}T\\
0
\end{array}\right)
\end{array}\right)+\\
 & -\left(\begin{array}{c}
\left(\begin{array}{c}
0\\
C^{-1}eM^{-1}PM^{-1}\left(p\Theta_{0}+e^{*}C^{-1}\lambda\Theta_{0}\right)\left(\Theta_{0}^{-1}\theta\right)
\end{array}\right)\\
\left(\begin{array}{c}
\left(\Theta_{0}p^{*}+\Theta_{0}\lambda^{*}C^{-1}e\right)M^{-1}PM^{-1}\left(p\Theta_{0}+e^{*}C^{-1}\lambda\Theta_{0}\right)\left(\Theta_{0}^{-1}\theta\right)\\
0
\end{array}\right)
\end{array}\right)\\
= & -W_{0}\left(\begin{array}{c}
v\\
T\\
\Theta_{0}^{-1}\theta\\
q
\end{array}\right)
\end{align*}

with $W_{0}$ given by{\small{}
\[
W_{0}\coloneqq\left(\begin{array}{cccc}
0 & 0 & 0 & 0\\
0 & C^{-1}eM^{-1}PM^{-1}e^{*}C^{-1} & C^{-1}eM^{-1}PM^{-1}\left(p\Theta_{0}+e^{*}C^{-1}\lambda\Theta_{0}\right) & 0\\
0 & \left(\Theta_{0}p^{*}+\Theta_{0}\lambda^{*}C^{-1}e\right)M^{-1}PM^{-1}e^{*}C^{-1}\; & \left(\Theta_{0}p^{*}+\Theta_{0}\lambda^{*}C^{-1}e\right)M^{-1}PM^{-1}\left(p\Theta_{0}+e^{*}C^{-1}\lambda\Theta_{0}\right) & 0\\
0 & 0 & 0 & 0
\end{array}\right).
\]
}{\small \par}

Summarizing, Equation (\ref{eq:quasi_static}) reads as
\begin{equation}
\partial_{0}\begin{array}{l}
\left(\left(\begin{array}{cccc}
\rho_{*} & 0 & 0 & 0\\
0 & M_{11} & M_{12} & 0\\
0 & M_{12}^{*} & M_{22} & 0\\
0 & 0 & 0 & \kappa_{1}
\end{array}\right)+\left(\begin{array}{cccc}
0 & 0 & 0 & 0\\
0 & 0 & 0 & 0\\
0 & 0 & 0 & 0\\
0 & 0 & 0 & \kappa_{0}^{-1}
\end{array}\right)+\left(\begin{array}{cccc}
0 & -\Div & 0 & 0\\
-\interior\Grad & 0 & 0 & 0\\
0 & 0 & 0 & \interior\dive\\
0 & 0 & \grad & 0
\end{array}\right)\right)\left(\begin{array}{c}
v\\
T\\
\Theta_{0}^{-1}\theta\\
q
\end{array}\right)=G,\end{array}\label{eq:static_2}
\end{equation}
 with
\begin{align}
M_{11} & \coloneqq C^{-1}-C^{-1}eM^{-1}PM^{-1}e^{*}C^{-1}\nonumber \\
M_{12} & \coloneqq C^{-1}\lambda\Theta_{0}-C^{-1}eM^{-1}PM^{-1}\left(p\Theta_{0}+e^{*}C^{-1}\lambda\Theta_{0}\right)\nonumber \\
 & =M_{11}\lambda\Theta_{0}-C^{-1}eM^{-1}PM^{-1}p\Theta_{0}\nonumber \\
M_{22} & \coloneqq\left(\gamma_{0}+\Theta_{0}\lambda^{*}C^{-1}\lambda\Theta_{0}\right)-\left(\Theta_{0}p^{*}+\Theta_{0}\lambda^{*}C^{-1}e\right)M^{-1}PM^{-1}\left(p\Theta_{0}+e^{*}C^{-1}\lambda\Theta_{0}\right)\label{eq:material-law}
\end{align}
and the right-hand side has to be adjusted to
\begin{equation}
G\coloneqq\left(\begin{array}{c}
F_{0}\\
F_{1}+C^{-1}eM^{-1}\left(M\interior\grad\right)\left|M\:\interior\grad\right|^{-2}\partial_{0}\psi\\
F_{4}+\left(\Theta_{0}p^{*}+\Theta_{0}\lambda^{*}C^{-1}e\right)M^{-1}\left(M\interior\grad\right)\left|M\:\interior\grad\right|^{-2}\partial_{0}\psi\\
F_{5}
\end{array}\right),\label{eq:right-hand-side}
\end{equation}
where we additionally assume that $\partial_{0}\psi\in D\left(\left|M\interior\grad\right|^{-2}\right).$
In the next theorem we provide sufficient conditions on the operators
involved in order to obtain a well-posedness result for (\ref{eq:static_2}).
\begin{thm}
Let $C,M,\rho_{\ast},\kappa_{1}$ be selfadjoint and non-negative
such that $C,M,\rho_{\ast},\nu\kappa_{1}+\kappa_{0}^{-1}\gg0$ for
sufficiently large $\nu$ and $P$ be an orthogonal projector. We
set $Q\coloneqq PM^{-1}e^{\ast}C^{-\frac{1}{2}}$ and assume that
\begin{align*}
1-Q^{\ast}Q & \gg0,\\
\gamma_{0}-\Theta_{0}p^{\ast}M^{-1}P\left(1-QQ^{\ast}\right)^{-1}PM^{-1}p\Theta_{0} & \gg0.
\end{align*}
Then, Equation (\ref{eq:static_2}) with $M_{11},M_{12,}M_{22}$ given
by (\ref{eq:material-law}) is well-posed. \end{thm}
\begin{proof}
We need to verify the solvability condition (\ref{eq:pos-def11}).
To do so, it suffices to consider the block operator sub-matrix
\[
\left(\begin{array}{cc}
M_{11} & M_{12}\\
M_{12}^{\ast} & M_{22}
\end{array}\right).
\]
Noting that $M_{11}=C^{-1}-C^{-\frac{1}{2}}Q^{\ast}QC^{-\frac{1}{2}}=C^{-\frac{1}{2}}\left(1-Q^{\ast}Q\right)C^{-\frac{1}{2}},$
we obtain that $M_{11}$ is boundedly invertible. Hence, by applying
a symmetric Gauss step, we are led to consider the matrix
\[
\left(\begin{array}{cc}
M_{11} & 0\\
0 & M_{22}-M_{12}^{\ast}M_{11}^{-1}M_{12}
\end{array}\right),
\]

which is strictly positive definite if and only if $M_{22}-M_{12}^{\ast}M_{11}^{-1}M_{12}\gg0.$
We have
\begin{align*}
M_{22} & =\left(\gamma_{0}+\Theta_{0}\lambda^{*}C^{-1}\lambda\Theta_{0}\right)-\left(\Theta_{0}p^{*}+\Theta_{0}\lambda^{*}C^{-1}e\right)M^{-1}PM^{-1}\left(p\Theta_{0}+e^{*}C^{-1}\lambda\Theta_{0}\right)\\
 & =\left(\gamma_{0}+\Theta_{0}\lambda^{*}C^{-1}\lambda\Theta_{0}\right)-\\
 & \quad-\left(\Theta_{0}p^{\ast}M^{-1}PM^{-1}p\Theta_{0}+2\Re\left(\Theta_{0}p^{\ast}M^{-1}QC^{-\frac{1}{2}}\lambda\Theta_{0}\right)+\Theta_{0}\lambda^{\ast}C^{-\frac{1}{2}}Q^{\ast}QC^{-\frac{1}{2}}\lambda\Theta_{0}\right)\\
 & =\gamma_{0}+\Theta_{0}\lambda^{\ast}\left(C^{-1}-C^{-\frac{1}{2}}Q^{\ast}QC^{-\frac{1}{2}}\right)\lambda\Theta_{0}-\\
 & \quad-\left(\Theta_{0}p^{\ast}M^{-1}PM^{-1}p\Theta_{0}+2\Re\left(\Theta_{0}p^{\ast}M^{-1}QC^{-\frac{1}{2}}\lambda\Theta_{0}\right)\right)\\
 & =\gamma_{0}+\Theta_{0}\lambda^{\ast}M_{11}\lambda\Theta_{0}-\left(\Theta_{0}p^{\ast}M^{-1}PM^{-1}p\Theta_{0}+2\Re\left(\Theta_{0}p^{\ast}M^{-1}QC^{-\frac{1}{2}}\lambda\Theta_{0}\right)\right)
\end{align*}
and
\[
M_{12}=M_{11}\lambda\Theta_{0}-C^{-\frac{1}{2}}Q^{\ast}M^{-1}p\Theta_{0}.
\]
Thus,
\begin{align*}
M_{12}^{\ast}M_{11}^{-1}M_{12} & =M_{12}^{\ast}\left(\lambda\Theta_{0}-M_{11}^{-1}C^{-\frac{1}{2}}Q^{\ast}M^{-1}p\Theta_{0}\right)\\
 & =\Theta_{0}\lambda^{\ast}M_{11}\lambda\Theta_{0}-2\Re\left(\Theta_{0}\lambda^{\ast}C^{-\frac{1}{2}}Q^{\ast}M^{-1}p\Theta_{0}\right)+\Theta_{0}p^{\ast}M^{-1}QC^{-\frac{1}{2}}M_{11}^{-1}C^{-\frac{1}{2}}Q^{\ast}M^{-1}p\Theta_{0}
\end{align*}
and hence, we get
\[
M_{22}-M_{12}^{\ast}M_{11}^{-1}M_{12}=\gamma_{0}-\Theta_{0}p^{\ast}\left(M^{-1}PM^{-1}+M^{-1}QC^{-\frac{1}{2}}M_{11}^{-1}C^{-\frac{1}{2}}Q^{\ast}M^{-1}\right)p\Theta_{0}.
\]
Using $M_{11}^{-1}=C^{\frac{1}{2}}\left(1-Q^{\ast}Q\right)^{-1}C^{\frac{1}{2}}$
we obtain
\[
QC^{-\frac{1}{2}}M_{11}^{-1}C^{-\frac{1}{2}}Q^{\ast}=Q\left(1-Q^{\ast}Q\right)^{-1}Q^{\ast}=-1+\left(1-QQ^{\ast}\right)^{-1}
\]
and since $Q=PQ$ we have
\[
QC^{-\frac{1}{2}}M_{11}^{-1}C^{-\frac{1}{2}}Q^{\ast}=-P+P\left(1-QQ^{\ast}\right)^{-1}P
\]
and thus,
\begin{align*}
M_{22}-M_{12}^{\ast}M_{11}^{-1}M_{12} & =\gamma_{0}-\Theta_{0}p^{\ast}\left(M^{-1}PM^{-1}+M^{-1}\left(-P+P\left(1-QQ^{\ast}\right)^{-1}P\right)M^{-1}\right)p\Theta_{0}\\
 & =\gamma_{0}-\Theta_{0}p^{\ast}M^{-1}P\left(1-QQ^{\ast}\right)^{-1}PM^{-1}p\Theta_{0},
\end{align*}
which is strictly positive definite by assumption.
\end{proof}
If we apply the latter theorem with $M=\sqrt{\varepsilon+e^{\ast}C^{-1}e}$,
$P=P_{\overline{M\interior\grad\left[L^{2}(\Omega)\right]}}$ and
$G$ as given in (\ref{eq:right-hand-side}), where we assume that
$\partial_{0}\psi\in D\left(\left|M\interior\grad\right|^{-2}\right)$,
we obtain a well-posedness result for the system described in (\ref{eq:quasi_static}),
where $E=-\interior\grad\varphi$, $D=\dive\psi$ and $D$ and $E$
are coupled by (\ref{eq:DandE}).

\section{Conclusion.}

The mathematical modelling of piezoelectric transducers can de-risk the development of new sensors and actuators and, coupled with the widespread availability of powerful computing facilities, there is an increasing reliance on this approach.  These models are typically used in inverse problems associated with obtaining a set of optimal design parameters for a desired set of sensor operating characteristics. It is therefore vital that these models are examined for their well-posedness.  There has been a steady body of work considering this problem over recent years and this paper has extended this to consider the case when the thermal effects are coupled to the piezoelectric equations.  This is motivated by the need to develop sensors that can operate at high temperatures in, for example, the non-destructive testing of heat exchanger surfaces in nuclear energy plants.   A modelling assumption that is often employed in order to reduce the size of the model is
the so-called quasi-electrostatic approach. Given the widespread
use of this approximation the well-posedness of the
quasi-electrostatic model of a piezoelectric material was also shown in this paper.


\begin{thebibliography}{1}

\bibitem{akamatsu2002} {\textsc Akamatsu, M. \& Nakamura, G.,}  Well-Posedness of Initial-Boundary Value Problems for Piezoelectric Equations, {\em Applicable Analysis,} \textbf{81}, 129--141, (2002).

\bibitem{algehyne2015} {\textsc Algehyne, E.A. \& Mulholland, A.J.,}  A finite element approach to modelling fractal ultrasonic transducers, {\em IMA J Appl Math,} first published online May 29, 2015, doi:10.1093/imamat/hxv012, (2015).

\bibitem{ammari2010} {\textsc Ammari, K. \& Nicaise, S.,} Stabilization of a piezoelectric system, {\em arXiv preprint arXiv:1005.2545},  (2010).

\bibitem{ionica2013} {\textsc Ionica, A., Nicusor, C. \& Andaluzia, M.,} Antiplane shear deformation of piezoelectric bodies in contact with a conductive support, {\em J Glob Optim,} 56, 103--119, (2013) .

\bibitem{barboteu2009} {\textsc Barboteu, M. \& Sofonea, M.,}  Solvability of a dynamic contact problem between a piezoelectric body and a conductive foundation, {\em Applied Mathematics and Computation,} 215, 2978--2991, (2009).

\bibitem{benaissa2015} {\textsc Benaissa, H., Essoufi,EL-H. \& Fakhar, R.,}  Existence results for unilateral contact problem with friction of thermo-electro-elasticity, {\em Appl. Math. Mech.-Engl. Ed.,} 36(7), 911--926, (2015).

\bibitem{bonaldi2015} {\textsc Bonaldi, F., Geymonat, G., \& Krasucki, F.,} Modelling of smart materials with thermal effects: dynamic and quasi-static evolution, {\em hal-01167897},  (2015).
    
\bibitem{Ciarlet97I} P.~G. Ciarlet. \newblock {\em Mathematical Elasticity, volume {II}: Theory of Plates},   volume~27 of {\em Studies in mathematics and its applications}. \newblock Elsevier Science B. V., Amsterdam, 1997.

\bibitem{chkadua2015} {\textsc Chkadua, G. \& Natroshvili, D.,}  Interaction of acoustic waves and piezoelectric structures,  {\em Mathematical Methods in the Applied Sciences}, 38(11), 2149--2170, (2015).

\bibitem{daros2000} {\textsc Daros, C.H. \& Antes, H.,}  The elastic motion of a transversely isotropic, piezoelectric solid caused by impulsive loading, {\em Z. angew. Math. Phys.,} 51, 397--418, (2000).

\bibitem{fang2013} {\textsc Fang, D. \& Liu, J.,}  Fracture Mechanics of Piezoelectric and Ferroelectric Solids, {\em Springer-Verlag Berlin}, (2013).

\bibitem{MR2318265} {\textsc Ferreira, M.~V. and  Perla~Menzala, G.,} Uniform stabilization of an electromagnetic-elasticity problem in exterior domains. {\em Discrete Contin. Dyn. Syst.}, 18(4):719--746, 2007.

\bibitem{imperiale2012} {\textsc Imperiale, S. \& Joly, P.,} Mathematical and Numerical Modelling of Piezoeletric Sensors, {\em ESAIM: M2AN} 46, 875--909, (2012).

\bibitem{jiang2000} {\textsc Jiang, L.Z.,}  Integral representation and Green's functions for 3D time-dependent thermo-piezoelectricity, {\em Int. J. Solids and Structures,} 37, 6155--6171, (2000).

\bibitem{kaltenbacher2006} {\textsc Kaltenbacher, B., Lahmer, T., Mohr, M.. \& Kaltenbacher, M.,}  PDE based determination of piezoelectric material tensors, {\em Euro. J. of Applied Mathematics,} 17, 383--416, (2006).

\bibitem{kapitonov2003} {\textsc Kapitonov, B. \& Perla Menzala, G.,} Uniform Stabilization and Exact Control of a Mulitlayered Piezoelectric Body, {\em Portugaliae Mathematica,} 60(4), 411--454,  (2003).

\bibitem{kapitonov2006} {\textsc Kapitonov, B., Miara, B. \& Perla Menzala, G.,}  Stabilization of a Layered Piezoelectric 3-D Body by Boundary Dissipation, {\em ESAIM: COCV,} 12, 198--215, (2006).

\bibitem{kapitonov2007} {\textsc Kapitonov, B., Miara, B. \& Perla Menzala, G.,}  Boundary Observation and Exact Control of a Quasi-Electrostatic Piezoelectric System in Multilayered Media, {\em SIAM J. Control Optim.} 46(3), 1080--1097, (2007).

\bibitem{lahmer2008} {\textsc Lahmer, T., Kaltenbacher, B. \& Schulz, V.,}  Optimal measurement selection for piezoelectric material tensor identification, {\em Inverse Problems in Science and Engineering,} 16(3), 369--387, (2008).

\bibitem{leugering2015} {\textsc Leugering, G. \& Nazarov, S.A.,}  The Eshelby Theorem and its Variants for Piezoelectric Media, {\em Arch. Rational Mech. Anal.} 215, 707--739, (2015).

\bibitem{mercier2005} {\textsc Mercier, D. \& Nicaise, S.,}  Existence, Uniqueness, and Regularity Results for Piezoelectric Systems, {\em SIAM J. Math. Anal.,} 37(2), 651--672, (2005).

\bibitem{miara2009} {\textsc Miara, B. \& Santos, M.L.,}  Energy decay in piezoelectric systems, {\em Applicable Analysis,} 88(7), 947--960, (2009).

\bibitem{migorski2009} {\textsc Migorski, S., Ochal, A. \& Sofonea, M.,}  Weak solvability of a piezoelectric contact problem, {\em Euro. J. Applied Mathematics,} 20, 145--167, (2009).

\bibitem{mindlin1974} {\textsc Mindlin, R.D.,}  Equations Of High Frequency Vibrations of Thermo-Piezoelectric Crystal Plates, {\em Intl. J. Solids \& Structures,} 10, 625--637, (1974).

\bibitem{mulholland2007} {\textsc Mulholland A.J., Ramadas S.N., O'Leary R.L., Parr A., Troge A., Pethrick R.A. \& Hayward G.,}  A Theoretical Analysis of a Piezoelectric Ultrasound Device with an Active Matching Layer, {\em Ultrasonics,} 47(1), 102--110, (2007).

\bibitem{mulholland2008} {\textsc Mulholland A.J., Ramadas S.N., O'Leary R.L., Parr A., A.Troge A., Pethrick R.A. \& Hayward G.,} Enhancing the performance of piezoelectric ultrasound transducers by the use of multiple matching layers, {\em IMA J. Appl. Maths.,} 73, 936--949, (2008).

\bibitem{orr2008a} {\textsc Orr L-A., Mulholland A.J., O'Leary R.L. \& Hayward G.,} Harmonic Analysis of Lossy Piezoelectric Composite Transducers using the Plane Wave Expansion Method, {\em Ultrasonics,} 48(8), 652--663, (2008) .

\bibitem{orr2007} {\textsc Orr L-A., Mulholland A.J., O'Leary R.L., Parr A, Pethrick R.A. \& Hayward G.,}  Theoretical Modelling of Frequency Dependent Elastic Loss in Composite Ultrasonic Transducers, {\em Ultrasonics,} 47(1), 130--137, (2007).

\bibitem{orr2008b} {\textsc Orr L-A., Mulholland A.J., O'Leary R.L., Pethrick R.A. \& Hayward G.,}  Analysis of Ultrasonic Transducers with Fractal Architecture, {\em Fractals,} 16(4), 222--349, (2008).

\bibitem{perla2014} {\textsc Perla Menzala, G. \& Sejje Su�rez, J.,}  Uniform stabilization of a thermopiezoelectric/piezomagnetic model, {\em J.. Math. Anal. Appl.,} 409, 56--73, (2014).

\bibitem{PDE_DeGruyter} R.~Picard and D.~F. McGhee. \newblock {\em Partial Differential Equations: A unified Hilbert Space   Approach}, volume~55 of {\em {De Gruyter Expositions in Mathematics}}. \newblock {De Gruyter. Berlin, New York. 518 p.}, 2011.

\bibitem{rafiee2013} {\textsc Rafiee, M., Mohammadi, M., Sobhani Aragh, B.. \& Yaghoobi, H.,}  Nonlinear free and forced thermo-electro-aero-elastic vibration and dynamic response of piezoelectric functionally graded laminated composite shells, Part I: Theory and analytical solutions, {\em Composite Structures,} 103, 179--187, (2013).
    
\bibitem{Rahmoune01121998} M.~Rahmoune, A.~Benjeddou, R.~Ohayon, and D.~Osmont. \newblock New thin piezoelectric plate models. \newblock {\em Journal of Intelligent Material Systems and Structures},   9(12):1017--1029, 1998.

\bibitem{sixto2013} {\textsc Sixto-Camacho, J.M., Bravo-Castillero, J., Brenner, R., Guinovart-D�az, R., Mechkour, H., Rodr�guez-Ramos, R., \& Sabina, F.J.,}  Asymptotic homogenization of periodic thermo-magneto-electro-elastic heterogeneous media, {\em Computers and Mathematics with Applications,} 66, 2056--2074, (2013).

\bibitem{tant2015a} {\textsc Tant,K.M.M., Mulholland, A.J., \& Gachagan, A.,}  A Model-Based Approach to Crack Sizing With Ultrasonic Arrays, {\em IEEE TUFFC,} 62(5), 915--926, (2015).

\bibitem{tant2015b} {\textsc Tant, K.M.M., Mulholland, A.J., Langer, M., \& Gachagan, A.,} A fractional Fourier transform analysis of the scattering of ultrasonic waves, {\em Proc. R. Soc. A,} 471 (2176), 20140958, DOI: 10.1098/rspa.2014.0958, (2015).

\bibitem{tramontana2015} {\textsc Tramontana, M., Gachagan, A., Nordon, A.., Littlejohn, D. \& Mulholland, A.J.,} System Modeling and Device Development for Passive Acoustic Monitoring of a Particulate-Liquid Process, {\em Sensors \& Actuators: A. Physical,} 228, 159--169, (2015) .

\bibitem{tucsnak1996} {\textsc Tucsnak, M.,}  Regularity and Exact Controllability for a Beam with Piezoelectric Actuator, {\em SIAM J. Control and Optimization,} 34(3), 922--930, (1996).

\bibitem{walker2015} {\textsc Walker, A.J. \& Mulholland, A.J.,}  A theoretical model of an ultrasonic transducer incorporating spherical resonators, {\em IMA J Appl Math,} first published online August 18, 2015 doi:10.1093/imamat/hxv023, (2015).

\bibitem{walker2011} {\textsc Walker A.J. \& Mulholland A.J.,}  Piezoelectric Ultrasonic Transducers with Fractal Geometry, {\em Fractals,} 19(4), 469--479, (2011).

\bibitem{walker2010} {\textsc Walker A.J., Mulholland A.J. \& Whitely S.,}  A Theoretical Model of an Electrostatic Ultrasonic Transducer incorporating Resonating Conduits, {\em IMA J. Appl. Maths,} 75(5), 796--810, (2010).

\bibitem{wynn2013} {\textsc Wynn, L.T., Truitt, A., Heim, I. \& Nima Mahmoodi, S.,}  Modeling and Response Analysis of Piezoelectric Flag In Wind Flow, {\em Proc. ASME 2013 Dynamic Systems and Control Conference DSCC2013, October 21-23, 2013, Palo Alto, California, USA}, DSCC2013-3912, (2013).

\bibitem{yuan2011} {\textsc Yuan, X. and Yang, F.,}  Energy Transfer in Pyroelectric Material, {\em Heat Conduction - Basic Research, Prof. Vyacheslav Vikhrenko (Ed.),} ISBN: 978-953-307-404-7, InTech, DOI: 10.5772/26053, (2011).
\end{thebibliography}
\end{document}